\documentclass[10pt]{article}
\usepackage{hyperref}
\hypersetup{
	colorlinks=true, 
	urlcolor=blue, 
	linkcolor=blue, 
	citecolor=red    
}
\usepackage{fancyhdr}
\usepackage{textcomp}
\usepackage[mathscr]{eucal}
\usepackage{amsmath,amsfonts,amsthm,amscd,amssymb,scalefnt}
\usepackage[margin=1in]{geometry}    
\usepackage{graphicx}                
\usepackage{subfig}
\usepackage{caption}
\usepackage{epstopdf}
\usepackage{multirow}
\usepackage[shortlabels]{enumitem}
\numberwithin{equation}{section}
\usepackage{cite}

\DeclareMathOperator{\sech}{sech}

\DeclareMathOperator{\arccosh}{arccosh}

\DeclareMathOperator{\sign}{sign}

\newtheorem{theorem}{Theorem}[section]
\newtheorem{corollary}[theorem]{Corollary}
\newtheorem{definition}[theorem]{Definition}
\newtheorem{remark}[theorem]{Remark}

\newtheorem{example}[theorem]{Example}
\newcommand{\thedate}{\today}
\makeatother

\begin{document}

\begin{center}
	\Large \bf{Some Characterizations of Timelike Rectifying Curves\\ in De Sitter 3-Space}
\end{center}
\vspace{10pt}
\centerline{\large  Mahmut MAK\footnote[1]{E-mail: mmak@ahievran.edu.tr\hfill\thedate}}
\vspace{10pt}
\centerline{\it K{\i}r\c{s}ehir Ahi Evran University, The Faculty of Arts and Sciences, Department of Mathematics, K{\i}r\c{s}ehir, Turkey.}
\vspace{10pt}

\begin{abstract} De Sitter space is a non-flat Lorentzian space form with positive constant curvature which plays an important role in the theory of relativity. In this paper, we define the notions of timelike rectifying curve and timelike conical surface in De Sitter 3-space as Lorentzian viewpoint. Moreover, we give some nice characterizations and results of a timelike  rectifying curves with respect to curve-hypersurface frame in De Sitter 3-space which is a three dimensional pseudo-sphere in Minkowski 4-space.
\end{abstract}

$~~~${\bf Keywords:} Rectifying curve, conical surface, geodesic, extremal curve, spiral curve.\smallskip

$~~~${\bf MSC2010:}$~~$53A35, 53C25.

\section{Introduction} \label{Sec:1}

In Euclidean 3-space $\mathbb{R}^3 $, let $ x:I \subseteq \mathbb{R} \to \mathbb{R}^3 $ be an unit speed regular curve with Frenet-Serret aparatus $ \left\{ {T,N,B,\kappa,\tau} \right\} $ where nonzero curvature $ \kappa $ and torsion $ \tau $ of the curve. At each point of the curve, the planes spanned by $ \left\{ {T,N} \right\} $, $ \left\{ {T,B} \right\} $ and $ \left\{ {N,B} \right\} $ are known as the osculating plane, the rectifying plane, and the normal plane, respectively. In  $\mathbb{R}^3$, it is well-known that a curve lies in a plane if its position vector lies in its osculating plane at each point; and it lies on a sphere if its position vector lies in its normal plane at each point.  In view of these basic facts, in $\mathbb{R}^3 $, the notion of rectifying curve  which is a space curves whose position vectors always lie in its rectifying planes, is firstly introduced by Bang-Yen Chen \cite{Chen2003}. Thus, the position vector $ x(s) $ of a rectifying curve satisfies the equation
\begin{eqnarray}\label{eq1_01}
x(s)-p = {c_1}(s)T(s) + {c_2}(s)B(s)
\end{eqnarray}
such that the fixed point $ p \in \mathbb{R}^3 $ for some differentiable functions $ c_1 $ and $ c_2 $ in arc length parameter $ s $ \cite{Chen2003,Kim1993}. It is known that a non-planar (twisted) curve in $\mathbb{R}^3 $ is a generalized helix if and only if the ratio $ \tau/\kappa $ is a nonzero constant on the curve. However, Chen show that for any given regular curve in $\mathbb{R}^3 $ is satisfied $ \tau/\kappa = c_1 s +c_2$ for some constants $ c_1\neq0 $ and $ c_2 $ in arc length parameter $ s $ iff the curve is congruent to a rectifying curve \cite{Chen2003}. Centrode is the path of the instantaneous center of rotation. It plays an important role in mechanics and kinematics. In  $\mathbb{R}^3 $, Darboux vector of a regular curve with a nonzero curvature is defined by $ D=\tau T + \kappa B $. However, the position vector of a rectifying curve is always in the direction of the Darboux vector which corresponds the instantaneous axis of rotation. Therefore, there is a hard relationship between the centrode and the rectifying curve. In this sense, Chen and Dillen give a relationship between rectifying curves and centrodes of space curves in \cite{Chen2005}. They study also rectifying curve as extremal curves and give a classification of curves with nonzero constant curvature and linear torsion in terms of spiral type rectifying curves in \cite{Chen2005}. After Chen's articles \cite{Chen2003,Chen2005}, rectifying curves and theirs characterizations are studied by many authors in different ambient spaces from various viewpoints. In this concept, some remarkable papers are \cite{Altunkaya2018space,Chen2017,Deshmukh2018,Ilarslan2008,Ilarslan2003,Izu2004New}. Moreover, the eq. (\ref{eq1_01}) means that "the straight line that passing through $ x(s) $ and the fixed point $ p $, is orthogonal to the principal normal line that starting at point $ x(s) $ with the direction of $ N(s) $". In this sense, Lucas and Yag\"{u}es give the concept of rectifying curves in three dimensional Spherical and hyperbolic space from the viewpoint of Riemannian Space Forms by using this idea in \cite{luc2015, luc2016}. 

It is well known that De Sitter space is a non-flat Lorentzian space form with positive constant curvature. Also, De Sitter 3-space is called a three dimensional pseudo-sphere in Minkowski 4-space as a semi-Riemannian hypersurface. Especially, De Sitter space is one of the vacuum solutions of the Einstein equations so it plays an important role in the theory of relativity.

In this study, as inspiration from \cite{luc2015, luc2016}, we introduce the notions of timelike rectifying curve with respect to curve-hypersurface frame and timelike conical surface in De Sitter 3-space as non-flat Lorentzian space form viewpoint. After, we give relationship between timelike rectifying curve and geodesic of timelike conical surface in De Sitter 3-space. Moreover, we obtain a nice characterization with respect to the ratio of geodesic torsion and geodesic curvature for timelike rectifying curves in De Sitter 3-space. However, we have a characterization which determines all timelike rectifying curve in De Sitter 3-space. Finally, in viewpoint of extremal curves, we give a corollary that a timelike curve in  De Sitter 3-space, which has non-zero constant geodesic curvature and linear geodesic torsion, congruent to a timelike rectifying curve, which is generated by a spiral type unit speed timelike curve with certain geodesic curvature in 2-dimensional pseudo-sphere, and vice versa.

\section{Preliminary} \label{Sec:2}

We give the fundamental notions for motivation to differential geometry of timelike curves and timelike surface in De Sitter 3-space and Minkowski 4-space. For more detail and background, see \cite{ChenLiang2016,HuangJie2019,luc2013,o1983semi}. Let $\mathbb{R}^4$ be a 4-dimensional real vector space and a scalar product in $\mathbb{R}^4 $ be defined by
\begin{eqnarray*}
	\langle{\boldsymbol{x}\,,\boldsymbol{y}\,}\rangle=-{x_1}{y_1}+{x_2}{y_2}+{x_3}{y_3}+{x_4}{y_4},
\end{eqnarray*}	
for any vectors $\boldsymbol{x}=(x_1,x_2,x_3,x_4),\,\textit{\textbf{y}}=(y_1,y_2,y_3,y_4)\in\mathbb{R}^4$. Then the pair $\left( {{\mathbb{R}^4},\left\langle , \right\rangle } \right) $ is called Minkowski 4-space (four-dimensional semi Euclidean space with index one), which is denoted by $\mathbb{R}_1^4$. We say that a nonzero vector $\boldsymbol{x} \in \mathbb{R}_1^4$ is called \textit{spacelike}, \textit{timelike} and \textit{null} if $\langle{\boldsymbol{x}\,,\boldsymbol{x}\,}\rangle>0\,$, $\langle{\boldsymbol{x}\,,\boldsymbol{x}\,}\rangle<0$ and $\langle{\boldsymbol{x}\,,\boldsymbol{x}\,}\rangle=0$, respectively. The \textit{norm} of  $\boldsymbol{x} \in \mathbb{R}_1^4$ is defined by $\left\| \boldsymbol{x} \right\| = \sqrt {\left| {\left\langle {\boldsymbol{x}\,,\boldsymbol{x}\,} \right\rangle } \right|} $. The \textit{signature} of a vector $ \boldsymbol{x} $ is defined by $ \sign( \boldsymbol{x})=1,\,0$ or $ -1 $ while  $ \boldsymbol{x}$ is spacelike, null or timelike, respectively.
The De Sitter 3-space is defined by
	\begin{eqnarray*}
		\mathbb{S}_1^3=\{\boldsymbol{x}\in\mathbb{R}_1^4\ \,|\,\,\langle{\boldsymbol{x}\,,\boldsymbol{x}\,}\rangle= 1\,\},
	\end{eqnarray*}	
which is a three dimensional unit pseudo-sphere (or a non-flat Lorentzian space form with positive constant curvature one) in  $ \mathbb{R}_1^4 $.

Let $\{{\boldsymbol{e}_1,\boldsymbol{e}_2,\boldsymbol{e}_3,\boldsymbol{e}_4}\}$ be the canonical basis of $\mathbb{R}_1^4$. Then the wedge product of any vectors $ \boldsymbol{x}=(x_1,x_2,x_3,x_4)$,  $\boldsymbol{y}=(y_1,y_2,y_3,y_4)$, $\boldsymbol{z}=(z_1,z_2,z_3,z_4) \in \mathbb{R}_1^4$ is given by
\begin{eqnarray*}\label{eq2_01}
	{\boldsymbol{x}}  \times  {\boldsymbol{y}}  \times  {\boldsymbol{z}} = \left| {\begin{array}{cccc}
			-{\boldsymbol{e}_1} & {\boldsymbol{e}_2} & {\boldsymbol{e}_3} & {\boldsymbol{e}_4} \\
			{x_1} & {x_2} & {x_3} & {x_4} \\
			{y_1} & {y_2} & {y_3} & {y_4} \\
			{z_1} & {z_2} & {z_3} & {z_4}
	\end{array}} \right|,
\end{eqnarray*}
where $\{{\boldsymbol{e}_1,\boldsymbol{e}_2,\boldsymbol{e}_3,\boldsymbol{e}_4}\}$ is the canonical basis of $\mathbb{R}_1^4$. Also, it is clear that
\begin{eqnarray}\label{eq2_02}
\langle \boldsymbol{w}\,,{\boldsymbol{x}}  \times  {\boldsymbol{y}}  \times  {\boldsymbol{z}}\rangle  = \det (\boldsymbol{w},\boldsymbol{x},\boldsymbol{y},\boldsymbol{z}),
\end{eqnarray}
for any $\boldsymbol{w}\in\mathbb{R}_1^4$. Hence, ${\boldsymbol{x}}  \times  {\boldsymbol{y}}  \times  {\boldsymbol{z}}$ is pseudo-orthogonal to each of the vectors $\boldsymbol{x},\boldsymbol{y}$ and $\boldsymbol{z}$. In the tangent space $ {T_q}\mathbb{S}_1^3 $ at any point $ q \in \mathbb{S}_1^3 $, we can give a cross product is denoted by  $ "\wedge" $ which is induced from the wedge product $ "\times" $ in $ \mathbb{R}_1^4. $ Let $ \boldsymbol{u},\,\boldsymbol{v} $ be tangent vectors (as considered column vectors of $ \mathbb{R}_1^4 $) in $ {T_q}\mathbb{S}_1^3 \subset \mathbb{R}_1^4$. Then the cross product $ \boldsymbol{u} \wedge \boldsymbol{v} $ in $ {T_q}\mathbb{S}_1^3 $ is given by
\begin{eqnarray}\label{eq2_02a}
\boldsymbol{u} \wedge \boldsymbol{v} = q \times \boldsymbol{u} \times \boldsymbol{v}.
\end{eqnarray}
By using (\ref{eq2_02}) and (\ref{eq2_02a}), it is easy to see that
\begin{eqnarray}\label{eq2_02b}
\left\langle {\boldsymbol{w},\boldsymbol{u} \wedge \boldsymbol{v}} \right\rangle  = -\det (q,\boldsymbol{u},\boldsymbol{v},\boldsymbol{w}),
\end{eqnarray}
for every $\boldsymbol{w} \in {T_q}\mathbb{S}_1^3 \subset \mathbb{R}_1^4$. Hence, we see that the orientations of a basis $\left\{ {{\bf{u}},{\bf{v}},{\bf{w}}} \right\} $ in $ {T_q}\mathbb{S}_1^3 $ and a basis $  \left\{ {q,{\bf{u}},{\bf{v}},{\bf{w}}} \right\} $ in $ \mathbb{R}_1^4 $ are opposite. Let $ M=\Phi(U) $ be regular surface which is identified by a immersion $ \Phi:U \subseteq {\mathbb{R}^2} \to \mathbb{S}_1^3 \subset \mathbb{R}_1^4 $ where $ U $ is open subset of $  {\mathbb{R}^2} $. Then, $ M $ is called spacelike or timelike surface in $\mathbb{S}_1^3$, if the tangent plane $ T_p{M} $ at any point $ p\in M$ is a spacelike subspace (it contains only spacelike vectors) or timelike subspace (i.e. it contains timelike, spacelike or null vectors) in $ \mathbb{R}_1^4 $, respectively.

Let $ M=\Phi(U) $ be non-degenerate (spacelike or timelike) surface in $ \mathbb{S}_1^3 $, and $ \xi $ be unit normal vector field of $ M $ such that $ \left\langle {\xi,\xi} \right\rangle  = \varepsilon  =  \pm 1 $. Then, for any differentiable vector fields  $ X,Y\in\mathfrak{X}(M) \subset \mathfrak{X}(\mathbb{S}_{1}^{3}) $, Gauss formulas of $ M $ are given by
\begin{eqnarray}
{\nabla^0}_X{Y} &=&{{\overline \nabla_X}Y - \left\langle {X,Y} \right\rangle \Phi},\label{eq2_02c}\\
{\overline \nabla}_X{Y} &=&{{\nabla_X}Y + \varepsilon \left\langle {\cal S}(X),Y \right\rangle \xi},\label{eq2_02d}
\end{eqnarray}
and  ${\cal S}:\mathfrak{X}(M) \to \mathfrak{X}(M)$ Weingarten map of $ M $ is given by
\begin{eqnarray}\label{eq2_02e}
{\cal S}(X)=-{\overline \nabla_X}\xi
\end{eqnarray} 
where Levi-Civita connections of ${\mathbb{R}_1^4} $, $\mathbb{S}_{1}^{3}$ and $ M $ are denoted by $ {\nabla }^0  $, $ {\overline \nabla } $ and $ \nabla $, respectively.

\begin{remark}\label{rmrk2_01}
	Let $ \Gamma$ be spacelike (or timelike) three dimensional hyperplane that passing through the origin in $\mathbb{R}_{1}^{4}$, then the surface $ \Gamma \cap \mathbb{S}_1^3 $ is congruent to unit sphere $\mathbb{S}^2$ ( or unit pseudo-sphere $ \mathbb{S}_1^2 $ ), which is a  spacelike (or timelike) totally geodesic surface in  $ \mathbb{S}_1^3 $. Moreover, let $ \Pi $  be spacelike (or timelike) plane that passing through the origin in $\mathbb{R}_{1}^{4}$, then  the curve $ \Pi \cap \mathbb{S}_1^3 $ is congruent to unit circle $ \mathbb{S}^1 $ ( or part of unit pseudo-circle $ \mathbb{S}_1^1 $ ), which is a  spacelike (or timelike) geodesic in  $ \mathbb{S}_1^3 $.
\end{remark}

Let $ p $ and $ q $ be distinct non-antipodal points of $\mathbb{S}_1^3 $, and $ \beta=\Pi \cap \mathbb{S}_1^3 $ be a geodesic that passing through the points $ p $ and $ q $ in  $ \mathbb{S}_1^3 $, where $  \Pi = Sp\left\{ {p,q} \right\}$. Then, for a vector $ \omega = q - \left\langle {p,q} \right\rangle p \in \mathbb{R}_{1}^{4}$, the parametrization of $ \beta = \beta(t)$ is given by
\begin{enumerate}[label={\upshape(\roman*)}, align=left, widest=i, leftmargin=*]\vspace{-3pt}
	\item If $ \left\langle {p,q} \right\rangle  > 1 $ (i.e. the angle $ \theta(p,q) =\arccosh\left( {\left\langle {p,q} \right\rangle } \right)$), then  $ \Pi $ is a timelike plane and $ \beta(t) = \cosh (t)p + \sinh (t)\frac{\omega}{{\left\| \omega \right\|}}$ (i.e a part of pseudo-circle) such that $ \sign(\omega)=-1$,
	\item If $ -1 < \left\langle {p,q} \right\rangle < 1 $ (i.e. the angle $ \theta(p,q) =\arccos\left( {\left\langle {p,q} \right\rangle } \right)$), then  $ \Pi $ is a spacelike plane and $ \beta(t) = \cos (t)p + \sin (t)\frac{\omega}{{\left\| \omega \right\|}}$ (i.e a circle) such that $ \sign(\omega) =1$,
	\item If  $ \left\langle {p,q} \right\rangle = 1 $, then  $ \Pi $ null plane and $ \beta(t) = p + t\,\omega $ (i.e a straight line) such that $ \sign(\omega)=0 $,	
	\item If $ \left\langle {p,q} \right\rangle  < -1 $, then there exists no geodesic joining $ p $ and $ q $. 
\end{enumerate}

Now, we consider the differential geometry of timelike regular curves in $\mathbb{S}_1^3$. Let $\alpha :I \to \mathbb{S}_1^3$ be a regular curve where $I$ is an open interval in $\mathbb{R}$. Then, the Gauss formula with respect to $ \alpha$ is given by
\begin{eqnarray}\label{eq2_03}
X'  \equiv  {{\nabla }^0 _{\alpha'}}X = {\overline \nabla  _{\alpha'}}X -\left\langle {\alpha',X} \right\rangle \alpha,
\end{eqnarray} 
for any differentiable vector field $ X\in\mathfrak{X}(\alpha(I))  \subset \mathfrak{X}(\mathbb{S}_{1}^{3}) $ along the curve $ \alpha $. We say that the regular curve $ \alpha $ is spacelike, null or timelike if  $ {\alpha}'(t)= {d\alpha}/{dt} $ is a spacelike vector, a null vector or a timelike vector, respectively, for any  $ t\in\textit{I} $. The regular curve $ \alpha $ is said to be a \textit{non-degenerate} (non-null) curve if  $ \alpha $ is a spacelike curve or a timelike curve. If $\alpha $ is a non-null curve, $ \alpha $ can be expressed with an arc length parametrization $ s=s(t) $.

Now, we assume that $ \alpha = \alpha(s) $ is a timelike unit speed curve in $\mathbb{S}_1^3$. Then the timelike unit tangent vector of $\alpha$ is given by $ {T_\alpha}(s) = \alpha'(s) $. We assume that  the assumption $T_{\alpha}'(s)-\alpha(s) \neq 0$, then the spacelike principal normal vector of $ \alpha $ is given by $ {N_\alpha}(s) = \frac{{{{\overline \nabla }_{{T_\alpha}(s)}}{T_\alpha}}(s)}{{\left\| {{{\overline \nabla }_{{T_\alpha}(s)}}{T_\alpha}(s)} \right\|}}$, which is pseudo-orthogonal to $\alpha(s)$ and $T_{\alpha}(s)$. Also, the spacelike binormal vector field of $ \alpha $ is given by  $ {B_\alpha}(s) = {T_\alpha}(s) \wedge {N_\alpha}(s) $ which is pseudo-orthogonal $ \alpha(s),\,{T_\alpha}(s)$ and $ {N_\alpha}(s) $. Thus, $ \left\{ {{T_\alpha}(s),\,{N_\alpha}(s),\,{B_\alpha}(s)} \right\} $ is called (intrinsic) Frenet Frame of non-geodesic timelike curve $ \alpha $ in ${T_{\alpha(s)}}\mathbb{S}_1^3 $ along the curve $ \alpha $. Also, from the equations (\ref{eq2_02a}) and (\ref{eq2_02b}), we see that \mbox{$B_{\alpha}(s) = \alpha(s) \times T_{\alpha}(s) \times N_{\alpha}(s)$}, and so we have pseudo-orthonormal frame $\{\alpha(s),T_{\alpha}(s),N_{\alpha}(s),B_{\alpha}(s)\}$ of $\mathbb{R}_1^4$ along $\alpha$. The frame is also called the \textit{curve- hypersurface frame} of timelike unit speed curve $ \alpha $ on $\mathbb{S}_1^3$.
By using Gauss formula (\ref{eq2_03}), under the assumptions $T_{\alpha}'(s)-\alpha(s) \neq 0$, the Frenet equations of $ \alpha $ in $ \mathbb{S}_1^3 $ is given by
\begin{eqnarray}\label{eq2_08}
{\overline \nabla  _{{T_\alpha}}}{T_\alpha} = {\kappa_g}{N_\alpha},\quad {\overline \nabla  _{{T_\alpha}}}{N_\alpha} = {\kappa_g}{T_\alpha} + {\tau_g}{B_\alpha},\quad {\overline \nabla  _{{T_\alpha}}}{B_\alpha} =  - {\tau_g}{N_\alpha},
\end{eqnarray}
where the geodesic curvature $ \kappa_{g} $ and the geodesic torsion $ \tau_g $ of $\alpha$ is given by
\begin{eqnarray}
{\kappa_g}(s)&=&\left\| {{{\overline \nabla  }_{T_{\alpha}(s)}}{T_{\alpha}}(s)} \right\| =  \left\| {T_{\alpha}'(s)-\alpha(s)} \right\|,\label{eq1_2_05a}\\
{\tau_g}(s)&=&{ \left\langle {{\overline \nabla  }_{{T_\alpha}(s)}}{N_\alpha}(s),{B_\alpha}(s) \right\rangle }=\frac{\det (\alpha(s),\alpha'(s),\alpha''(s),\alpha'''(s))}{{{{({\kappa_g}(s))}^2}}}\label{eq1_2_05b}.
\end{eqnarray}
By using Gauss formula (\ref{eq2_03}), Frenet equations of $ \alpha $ is also given by
\begin{eqnarray}\label{eq2_06}
\nabla _{{T_\alpha}}^0{T_\alpha} = \alpha + {\kappa_g}{N_\alpha},\quad \nabla _{{T_\alpha}}^0{N_\alpha} = {\kappa_g}{T_\alpha} + {\tau_g}{B_\alpha},\quad \nabla _{{T_\alpha}}^0{B_\alpha} =  - {\tau_g}{N_\alpha},
\end{eqnarray}
with respect to Levi-Civita connection of $ \mathbb{R}_1^4 $.

A non-degenerate curve in $\mathbb{S}_1^3$ is a geodesic iff its geodesic curvature $ \kappa_{g} $ is zero at all points. By using  (\ref{eq2_03}) and (\ref{eq1_2_05a}), we see that the assumption $T_{\alpha}'(s)-\alpha(s) \neq 0$ (or equivalently $< \alpha''(s),\alpha''(s) > \neq 1$) corresponds to the curve $ \alpha $ is not a geodesic (i.e. $\kappa_{g}\neq0 $). A non-degenerate curve in $\mathbb{S}_1^3$ is a planar curve iff it lies in a non-degenerate two-dimensional totally geodesic surface (i.e. geodesic torsion $ \tau_g $ is zero at all point) in  $\mathbb{S}_1^3$. By Remark \ref{rmrk2_01}, we say that timelike planar curve in $ \mathbb{S}_1^3 $, lies fully in two-dimensional unit pseudo-sphere $ \mathbb{S}_1^2 \subset \mathbb{S}_1^3 $. 	

\begin{remark}\label{rmrk2_02}
	Let $ M=\Phi(U) $ be non-degenerate surface in $ \mathbb{S}_1^3 $ with the unit normal vector field $ \xi $, and $ {T_\beta} $ be the unit tangent vector field of a non-degenerate unit speed curve $\beta$ which lies on the surface $ M $. Then, $ \beta $ is a geodesic of $ M $ iff $\nabla_{T_\beta}{T_\beta}=0$. By using (\ref{eq2_02d}), we conclude that $\overline \nabla_{T_\beta}{T_\beta}$ is parallel to $ \xi $. Namely, $\overline \nabla_{T_\beta}{T_\beta}$ is orthogonal to the surface $ M $.
\end{remark}

Let $ p\in \mathbb{S}_1^3 $ and $ w\in {T_{p}\mathbb{S}_1^3} $. The exponential map $ {\exp _p}:{T_p}\mathbb{S}_1^3 \to \mathbb{S}_1^3 $ at $ p\in \mathbb{S}_1^3$ is defined by $ {\exp_p}(w) = {\beta_w}(1) $ where $ {\beta_w}:[0,\infty) \to \mathbb{S}_1^3 $ is the constant speed geodesic starting from $ p $ with the initial velocity $ {\beta'_{w}}(0) = w $. Also, the property $ {\exp _p}(tw) = {\beta_{tw}}(1) = {\beta_w}(t) $ is satisfied for any $ t\in \mathbb{R} $. In that case, for any point $ \alpha(s) $ in the timelike curve $ \alpha $ , the spacelike principal normal geodesic in $ \mathbb{S}_1^3 $
starting at $ \alpha(s)$ is defined by the geodesic curve
\begin{eqnarray*}
	{\exp _{\alpha(s)}}(t{N_\alpha}(s)) = \cos (u)\alpha(s) + \sin (u){N_\alpha}(s),\, t\in\mathbb{R}.
\end{eqnarray*}

Let's remind an important property of the parallel transport. A vector field which is pseudo-orthogonal to tangent vector of a geodesic $ {\beta_w}(t)= {\exp _p}(tw) $ in  $ \mathbb{S}_1^3 $, is invariant under parallel transport $ \cal{P} $ along the geodesic. Thus, the parallel transport $ {\cal{P}} $ from $ \alpha(s) $ to $ {\exp _{\alpha(s)}}(t{N_\alpha}(s)) $ along the spacelike principal normal geodesic satisfies the following statements:
\begin{eqnarray*}
	{\cal{P}}({N_\alpha}(s)) =  - \sin (u)\alpha(s) + \cos (u){N_\alpha}(s),
\end{eqnarray*}
and 
\begin{eqnarray*}
	{\cal{P}}({T_\alpha}(s)) = {T_\alpha}(s),~{\cal{P}}({B_\alpha}(s)) = {B_\alpha}(s).
\end{eqnarray*}

\section{Timelike rectifying curves in $\mathbb{S}_1^3 $} \label{Sec:3}

In this section, we give the notions of timelike rectifying curve and timelike conical surface  $ \mathbb{S}_1^3 $. After, we obtain some characterizations for timelike rectifying curves in $ \mathbb{S}_1^3 $.

\begin{definition}\label{def3_01}
	Let $ \alpha=\alpha(s) $ be a timelike non-geodesic unit speed curve in $ \mathbb{S}_1^3 $ and $ p $ be a fixed point in $\mathbb{S}_1^3 $ such that $ \left\{ { \pm p} \right\} \notin {\rm{Im}}(\alpha) $. Then  $ \alpha$ is called a timelike rectifying curve  in $ \mathbb{S}_1^3 $ iff the geodesics in $\mathbb{S}_1^3 $ that passing through $ p $ and $ \alpha(s) $ are pseudo-orthogonal to the spacelike principal normal geodesics at $ \alpha(s) $ for every $ s $.
\end{definition}

From the Definition \ref{def3_01}, the geodesics that passing through $ p $ and $ \alpha(s) $ are tangent to the timelike rectifying planes $ S_p\left\{T_\alpha{(s)},B_\alpha{(s)}\right\}$ of $ \alpha $ in  ${T_{\alpha(s)}}{\mathbb{S}_{1}^3 } \subset \mathbb{R}_1^4$. Namely, $ \frac{d}{dt}\beta_s{(t)} \in S_p\left\{T_\alpha{(s)},B_\alpha{(s)}\right\}$, and so it is easily seen that rectifying condition is given by 
\begin{eqnarray*}
	\left\langle \frac{d}{dt}\beta_s{(t)},N_\alpha{(s)}\right\rangle =0,\,t\in \mathbb{R},
\end{eqnarray*}
where $	\beta_s(t) $ is a geodesic that passing through $ p $ and $ \alpha(s) $ such that
\begin{eqnarray*}
	\beta_s(t)=\exp_p(t\,\alpha(s))=\cos(t)p+\sin(t)\alpha(s),\, t\in\mathbb{R}.
\end{eqnarray*}

\begin{remark}\label{rmrk_04}
	The principal normal vector field of any non-degenerate planar curve in $ \mathbb{S}_1^3 $ is tangent to $ \mathbb{S}_1^2 \subset \mathbb{S}_1^3 $ or $ {\mathbb{S}^2} \subset \mathbb{S}_1^3 $ since its geodesic torsion is zero at all points. So the tangent vector of geodesic connecting the planar curve with any point which is not element to the curve's image is orthogonal to principal normal vector field of the planar curve (see \cite{luc2015}). Hence, we say that every non-degenerate planar curve in $ \mathbb{S}_1^3 $ is a rectifying curve. From now on, we will assume that the curve $ \alpha $ is a timelike non-geodesic (i.e. $ \kappa_{g}>0 $) and non-planar curve (i.e. $ \tau_g\neq0 $) in $ \mathbb{S}_1^3 $.	
\end{remark}

\begin{definition}\label{def3_02}
	Let $ M=\Phi(U) $ be a timelike regular surface in $\mathbb{S}_1^3 $ via timelike immersion $ \Phi:U \to \Phi(U) \subset \mathbb{S}_1^3 $  such that open subset $ U \subseteq {\mathbb{R}^2} $. Then $ M $ is called a timelike conical surface in $ \mathbb{S}_1^3 $ if and only if $ M $ is constructed by the union of all the geodesics that pass through a fixed point (the apex) $ p\in \mathbb{S}_1^3 $ and any point of some regular timelike curve (the directrix) that does not contain the apex. Also, each of those geodesics is called a generatrix of the surface.
\end{definition}

Let the fixed point $p \in \mathbb{S}_1^3$ be apex of $ M $, and $ \gamma $ be the directrix of $ M $, which is a timelike unit speed curve in $ \mathbb{S}_1^2 \subset {T_p}\mathbb{S}_1^3 \subset \mathbb{R}_1^4 $. Then the parametrization of timelike conical surface $ M$ is given by 
\begin{eqnarray}{\label{eq3_01}}
\Phi(u,v) = {\exp _p}(v\gamma(u)) = \cos (v)p + \sin (v)\gamma(u),\quad 0 < v.
\end{eqnarray}

Let $ \cal{P} $ be the parallel transport along the geodesic $\beta_u(t)={\exp _p}(t\,\gamma(u))=\Phi(u,t)$ that passing through $ p $ and $ \gamma(u) $. Also, it is known that a vector field which is pseudo-orthogonal to $ \frac{d}{{dt}}{\beta_u}(t) $, is invariant under  parallel transport $ \cal{P} $ along the geodesic $ {\beta_u} $. Then the timelike tangent plane of $ M $ is spanned by the timelike vectors ${\Phi_u}(u,v)$ and spacelike vector ${\Phi_v}(u,v) $ is given by
\begin{eqnarray}{\label{eq3_02}}
{\Phi_u}(u,v)&=&\sin(v)\gamma'(u) = \sin(v){\cal{P}}(\gamma'(u)), \\ {\label{eq3_03}}
{\Phi_v}(u,v)&=&{\rm{ - }}\sin(v)p + \cos(v)\gamma(u) = {\cal{P}}(\gamma(u)).
\end{eqnarray}

Coefficients of the first fundamental form of $ M $ is 
\begin{eqnarray}\label{eq3_02a}
E=\left\langle {\Phi _u},{\Phi _u}\right\rangle  = -{\sin ^2}(v) ,\quad F= \left\langle {\Phi _u},{\Phi _v}\right\rangle  = 0,\quad
G= \left\langle {\Phi _v},{\Phi _v}\right\rangle  = 1.
\end{eqnarray}

The spacelike unit normal vector field $\xi(u,v)$ of $ M $ is given by
\begin{eqnarray}{\label{eq3_04}}
\xi(u,v) = \frac{{{\Phi_u} \wedge {\Phi_v}}}{{\left\| {{\Phi_u} \wedge {\Phi_v}} \right\|}}\left( {u,v} \right) = \frac{{\sin \left( v \right)}}{{\sqrt {\left| {EG - {F^2}} \right|} }}\left( {{\cal{P}}\left( {\gamma'(u)} \right) \wedge {\cal{P}}\left( {\gamma(u)} \right)} \right) =  - {\cal{P}}\left( {{N_\gamma}(u)} \right) =  - {N_\gamma}(u),
\end{eqnarray}
where $ {N_\gamma}(u) = \gamma(u) \wedge \gamma'(u) $ is a spacelike unit vector field tangent to $ \mathbb{S}_1^2 \subset {T_p}\mathbb{S}_1^3 $. 

Let ${\cal S}:\mathfrak{X}(M) \to \mathfrak{X}(M)$ be the Weingarten map of $ M $, and  $ {\kappa_\gamma} = \det (\gamma,\gamma',\gamma'',p) $ be the geodesic curvature of the timelike unit speed curve $ \gamma $ with respect to Sabban frame (curve-surface frame) $ \left\{ {\gamma,{T_\gamma} = \gamma',{N_\gamma} = \gamma \wedge {T_\gamma}} \right\} $ in $ \mathbb{S}_1^2 \subset {T_p}\mathbb{S}_1^3 $ such that 
\begin{eqnarray}\label{eq3_05}
\left\{ 
\begin{array}{l}
{\overline \nabla_{{T_\gamma}}}{T_\gamma}={\kappa_\gamma}{N_\gamma}\\
{\overline \nabla_{{T_\gamma}}}{N_\gamma}={\kappa_\gamma}{T_\gamma}
\end{array}
\right..
\end{eqnarray}
Then, we obtain following equations 
\begin{eqnarray*}
	{\cal S}({\Phi_u}) &=&  - {\overline \nabla  _{{\Phi_u}}}\xi = {\overline \nabla_{{T_\gamma}(u)}}{N_\gamma}(u) = \frac{{\kappa_{\gamma}(u)}}{{\sin (v)}}{\Phi_u}, \\
	{\cal S}({\Phi_v}) &=&  - {\overline \nabla_{{\Phi_v}}}\xi =  - \nabla _{{\Phi_v}}^0{\xi} = 0.
\end{eqnarray*}
by using (\ref{eq2_02e}), (\ref{eq3_02}), (\ref{eq3_04}) and (\ref{eq3_05}). Thus, the Gaussian curvature $ K $ and the mean curvature $ H $ of $ M $ is 
\begin{eqnarray*}
	K = {K_e} + \det ({\cal S}) = 1,\quad H = \frac{1}{2}tr({\cal S}) = \frac{{{\kappa_\gamma}(u)}}{{2\sin (v)}}
\end{eqnarray*}
where $ K_e $ is extrinsic Gaussian curvature of $ M $. Moreover, we obtain the following equations
\begin{eqnarray}
\overline{\nabla}_{\Phi_u}{\Phi_u}&=& \sin (v)\cos (v){\Phi_v} - {\kappa_\gamma}(u)\sin (v)\xi,{\label{eq3_07}} \\ 
\overline{\nabla}_{\Phi_u}{\Phi_v}&=&\overline{\nabla}_{\Phi_v}{\Phi_u}=\cot(v){\Phi_u},{\label{eq3_08}} \\ 
\overline{\nabla}_{\Phi_v}{\Phi_v}&=&0,{\label{eq3_09}}
\end{eqnarray}
by using (\ref{eq2_02c}), (\ref{eq3_02}) and (\ref{eq3_03}).

Now, we give the relationship between timelike rectifying curves and timelike conical surface in $ \mathbb{S}_1^3 $. 

\begin{theorem}\label{teo3_01}
	Let $ \alpha = \alpha(s) $ be a timelike unit speed curve in $ \mathbb{S}_1^3 $, and $ p \in\mathbb{S}_1^3$ be a fixed point such that $ p \notin {\rm Im}(\alpha) $. Then, $ \alpha $ is a
	timelike rectifying curve if and only if $ \alpha $ is a geodesic of the timelike conical surface $ M $ with apex $ p $ and timelike director curve $ \gamma $ which is a timelike unit speed  curve in 2-dimensional pseudo-sphere $ \mathbb{S}_1^2 \subset T_{p}{\mathbb{S}_1^3} $.
\end{theorem}

\begin{proof}
	Let $ \alpha = \alpha(s) $ be a timelike unit speed rectifying curve in $ \mathbb{S}_1^3 $, and $ p \in\mathbb{S}_1^3$ be a fixed point such that $ p \notin {\rm Im}(\alpha) $. Then, the parametrization of $ \alpha $ can be given by  $ \alpha(s) = {\exp _p}(v(s)\gamma(u(s))) $ for some functions $ u(s) $ and $ v(s) $ such that a timelike unit speed curve $ \gamma=\gamma(u) $ in $ \mathbb{S}_1^2 \subset T_{p}{\mathbb{S}_1^3} $. By using (\ref{eq3_01}), $\alpha(s) = \Phi(u(s),v(s)) $ is a timelike curve on timelike conical surface $ M $ which is determined by the apex $ p $ and the directrix $ \gamma $. Hence,  the spacelike geodesic that passing through $ p $ and $ \alpha(s) $ is given by $ \beta_s(t)=\Phi(u(s),t) $, and also we get the rectifying condition
	\begin{eqnarray}\label{eq3_10}
	\left\langle {{T_{{\beta_s}}}(v(s)),{N_\alpha}(s)} \right\rangle  = 0,
	\end{eqnarray}
	where $ T_{\beta_s} $ is the spacelike unit tangent vector of $ \beta_s $. So, the timelike tangent plane at point $ \alpha(s) $ of $ M $ consists of tangent vectors to the generatrix including the point $ \alpha(s) $ of $ M $. Namely, $ T_{\alpha(s)}M =\{ T_{{\beta_s}}(v(s)), T_{\alpha}(s) \} $. Then, by using (\ref{eq3_10}, we obtain that $ N_{\alpha}(s) $ is orthogonal to $ M $. Thus, ${\overline \nabla _{{T_\alpha}}}{T_\alpha} $ is parallel unit normal vector field of $ M $ by using (\ref{eq2_08}), and so $ \alpha $ is a geodesic of $ M $ by Remark \ref{rmrk2_02}.
	
	On the other hand, let  $ \alpha $ be a timelike unit speed geodesic of the timelike conical surface $ M $ with apex $ p $, and its parametrization be given by $ \alpha(s) = \Phi(u(s),v(s)) $ with arc length parameter $ s $. Then the spacelike principal normal vector $ N_{\alpha}(s) $ of $ \alpha $ is orthogonal to timelike surface $ M $. From here, $ N_{\alpha}(s) $ is also orthogonal to $ \beta_s $, which is the spacelike unit speed generatrix that passing through  $ p=\beta_s(0) $ and $ \alpha(s)=\beta_{s}(v(s)) $. Thus, we get that the rectifying condition (\ref{eq3_10}), and so $ \alpha $ is a timelike rectifying curve.
\end{proof} 

Now, we give a characterization with respect to the ratio of geodesic torsion and geodesic curvature for timelike rectifying curves in $ \mathbb{S}_1^3 $.

Let $\alpha:I\to M\subset\mathbb{S}_{1}^3,\,\,\alpha(s) = \Phi(u(s),v(s)) $ be a timelike unit speed curve in a timelike conical surface $ M$ which is given by the parametrization (\ref{eq3_01}) such that some differentiable functions $ u=u(s) $ and $ v=v(s) $. By using (\ref{eq3_02}) and (\ref{eq3_03}), we obtain that 
\begin{eqnarray}\label{eq3_11}
- 1 = \left\langle {{T_\alpha},{T_\alpha}} \right\rangle  =  - {(u')^2}{\sin ^2}(v) + {(v')^2}.
\end{eqnarray}
Moreover, from (\ref{eq2_02c}), we write $ {\overline \nabla  _{{T_\alpha}}}{T_\alpha} = {\nabla ^0}_{{T_\alpha}}{T_\alpha} - \Phi $ for $ T_\alpha \in \mathfrak{X}(M)$. Thus, we have the following equation
\begin{eqnarray}\label{eq3_12}
{\overline \nabla _{{T_\alpha}}}{T_\alpha}=(u'' + 2{u' }{v' }\cot (v)) {\Phi_u} + ({v'' + {{({u' })}^2}\sin (v)\cos (v)}) {\Phi_v} - {({u' })^2}{\kappa_\gamma}(u) \sin(v)\xi,
\end{eqnarray}
by using (\ref{eq3_02a}), (\ref{eq3_07}), (\ref{eq3_08}) and (\ref{eq3_09}).

Let $ \alpha $ be a geodesic in timelike conical surface $ M $ (namely, a timelike rectifying curve in $ \mathbb{S}_1^3 $). Then, it is easily seen that spacelike principal normal $ N_\alpha $ of $ \alpha $ is parallel to spacelike unit normal vector field $ \xi $ of $ M $ from Theorem \ref{teo3_01}, and so we obtain the following differential equation system
\begin{eqnarray}
u'' + 2u'v'\cot (v) &=& 0, \label{eq3_13}\\
v'' + {(u')^2}v'\sin (v)\cos (v) &=& 0, \label{eq3_14}\\
- {(u')^2}{\kappa_\gamma}(u)\sin (v)&=& {\kappa_g} > 0. \label{eq3_15}
\end{eqnarray}
with respect to functions $ u(s) $ and $ v(s) $.

If we take $ f(s) =\cos(v(s)) $, then we get the following differential equation
\begin{eqnarray*}
	f''(s) - f(s) = 0,
\end{eqnarray*}
by using  (\ref{eq3_11}) and (\ref{eq3_14}). Therefore, the solution is given by
\begin{eqnarray}\label{eq3_16}
f(s)={\lambda_1 }\sinh(s+s_0) + {\lambda_2} \cosh(s+s_0),
\end{eqnarray}
and so
\begin{eqnarray}\label{eq3_17}
v(s)=\arccos ({\lambda_1} \sinh(s+s_0) + {\lambda_2} \cosh(s+s_0)),
\end{eqnarray}
where some constants $ {\lambda_1},{\lambda_2} $ and $ s_0 $. Since $ N_{\alpha}(s) $ is parallel to $ \xi(u(s),v(s)) $, In without loss of generality, we can write that  
\begin{eqnarray*}\label{eq3_18}
	{B_\alpha}(s) = {\lambda}(s){\Phi_u}(u(s),v(s)) + {\mu}(s){\Phi_v}(u(s),v(s)),
\end{eqnarray*}
where $ {\lambda} = \frac{{ < {B_\alpha},{\Phi_u} > }}{E} = \frac{{v'}}{{\sin (v)}} $ and $ {\mu} = \frac{{ < {B_\alpha},{\Phi_v} > }}{G} =  - u'\sin (v) $. After straightforward calculation, we get that
\begin{eqnarray*}
	\overline{\nabla}_{T_\alpha}{B_\alpha}={\lambda}^{\prime}\Phi_u+{\mu}^{\prime}\Phi_v+({\lambda}{u'})\overline{\nabla}_{\Phi_u}{\Phi_u}+({\lambda}{v'}+{\mu}{u'})\overline{\nabla}_{\Phi_u}{\Phi_v}.
\end{eqnarray*}
By using (\ref{eq2_08}), (\ref{eq3_07}) and (\ref{eq3_08}), we obtain that the geodesic torsion which is given by
\begin{eqnarray*}
	\tau_g={u'}{v'} \kappa_\gamma{(u)}.
\end{eqnarray*}
From the last equation and (\ref{eq3_15}), we get that 
\begin{eqnarray}\label{eq3_19}
\frac{{{\tau_g}}}{{{\kappa_g}}} = \frac{{v'}}{{u'\sin(v)}}.
\end{eqnarray}
Now, by using (\ref{eq3_13}), we obtain 
\begin{eqnarray*}
	u''{\sin ^2}(v) + 2u'v'\sin (v)\cos (v) = 0,
\end{eqnarray*}
and after changing of variable, we have
\begin{eqnarray}\label{eq3_20}
u'{\sin ^2}(v) = {c},
\end{eqnarray}
for a nonzero constant $ {c} $.
If we consider together the equations (\ref{eq3_11}), (\ref{eq3_16}), (\ref{eq3_17}) and (\ref{eq3_20}), then we get the relation
\begin{eqnarray}\label{eq3_21}
{{c}^2} = {{\lambda_1}^2} - {{\lambda_2}^2} +1,
\end{eqnarray}
for the constants $ {\lambda_1},\,{\lambda_2} $ and $ {c} $. Finally, after required calculations by using (\ref{eq3_17}), (\ref{eq3_19}) and (\ref{eq3_20}), we obtain that 
\begin{eqnarray}\label{eq3_22}
\frac{{{\tau_g}}}{{{\kappa_g}}}(s) = {{\mu_1}}\sinh (s + {s_0}) + {{\mu_2}}\cosh (s + {s_0}),
\end{eqnarray}
for some constants $ {\mu_1}=\frac{-{\lambda_2}}{{c}}$, ${\mu_2}=\frac{-{\lambda_1}}{{c}}$  and $ s_0 $ such that $ {{\mu_2}}^2 - {{\mu_1}}^2 < 1 $.

On the other hand, let $ \alpha = \alpha(s) $ be a timelike unit speed  curve in $ \mathbb{S}_1^3 $ whose geodesic curvatures  satisfying the equation (\ref{eq3_22}) for some constants $ {\mu_1} $ and ${\mu_2}$ such that $ {{\mu_2}}^2 - {{\mu_1}}^2 < 1 $. Let $ {c} $ be a nonzero constant such that
\begin{eqnarray*}
	{{c}^2} = \frac{1}{{{{\mu_1}}^2 - {{\mu_2}}^2 + 1}},
\end{eqnarray*}
and define two constants $ {\lambda_1} =  - {{\mu_2}}{c} $ and $ {\lambda_2} =  - {{\mu_1}}{c} $. Let the function $ v(s) $ be defined by (\ref{eq3_17}) and the function $ u(s) $ be a solution of (\ref{eq3_20}), which is given by
\begin{eqnarray*}
	u(s) = {\tanh ^{ - 1}}\left( {\frac{{{\lambda_1}{\lambda_2}}}{{c}} + \left( {\frac{{1 + {{\lambda_1}^2}}}{{c}}} \right)\tanh (s + {s_0})} \right).
\end{eqnarray*}
Now, let $ M $ be the timelike conical surface with apex  
$ p \notin {\rm Im}(\alpha) \subset \mathbb{S}_1^3$ and timelike director curve $ \gamma$ which is a timelike unit speed curve in $ \mathbb{S}_1^2 \subset T_{p}{\mathbb{S}_1^3} $ such that geodesic curvature  $ \kappa_\gamma $ of $ \gamma $ is given by (\ref{eq3_15}). Then, we consider a timelike unit speed curve $\widetilde{\alpha} (s) = {\exp _p}(v(s)\gamma(u(s)))=\Phi(u(s),v(s)) $ in $ M$, which is given by the parametrization (\ref{eq3_01}) for some differentiable functions $ u=u(s) $ and $ v=v(s) $. It is easily seen that $\widetilde{\alpha}$ is a geodesic of $ M $, whose geodesic curvatures are same with geodesic curvatures of the curve $ \alpha $, and so  $ \alpha $ is congruent to a geodesic in a timelike conical surface. Thus, we show that the equation (\ref{eq3_22}) determines the timelike curves in $ \mathbb{S}_1^3 $ that are geodesics in a timelike conical surface whose parametrization is $ \Phi(u(s),v(s)) $. As a result, we say that $ \alpha $ is congruent to a timelike rectifying curve in $ \mathbb{S}_1^3 $ by Theorem \ref{teo3_01}.

Consequently, we obtain the following characterization for timelike rectifying curves in $ \mathbb{S}_1^3 $ with respect to the ratio of the geodesic curvatures. 

\begin{theorem}\label{teo3_02}
	Let $ \alpha = \alpha(s) $ be a timelike unit speed curve in $ \mathbb{S}_1^3 $ with geodesic curvature $ \kappa_g $ and geodesic torsion $ \tau_g $. Then $ \alpha $ is congruent to a timelike rectifying curve iff the ratio of geodesic torsion and geodesic curvature of the timelike curve is given by
	\begin{eqnarray*}
		\frac{{{\tau_g}}}{{{\kappa_g}}}(s) = {{\mu_1}}\sinh (s + {s_0}) + {{\mu_2}}\cosh (s + {s_0}),
	\end{eqnarray*}
	for some constants $ {\mu_1},{\mu_2} $ and $ s_0 $ such that $ {{\mu_2}}^2 - {{\mu_1}}^2 < 1 $.	
\end{theorem}

Now, we give some characterizations for timelike rectifying curves in $ \mathbb{S}_1^3 $.

\begin{theorem}\label{teo3_03}
	Let $ \alpha = \alpha(s) $ be a timelike unit speed curve in $ \mathbb{S}_1^3 $ and $ p $ be the fixed point in $ \mathbb{S}_1^3 $ such that $ p \notin {\rm Im}(\alpha) $. Then the following statements are equivalent:	
	\begin{enumerate}[label={\upshape(\roman*)}, align=left, widest=i, leftmargin=*]
		\item $ \alpha $ is a timelike rectifying curve.
		\item $ p^ \bot $ is the component of $ p $ which is orthogonal to $ T_{\alpha} $ in  $ \mathbb{S}_1^3 $ such that
		\begin{eqnarray}
		< p,{T_\alpha}(s) > &=&{{n_1}}\sinh \left( {s + {s_0}} \right) + {{n_2}}\cosh \left( {s + {s_0}} \right),\label{eq3_ek02}\\
		|{p^ \bot }|^2 &=& {n^2},\label{eq3_ek03}	
		\end{eqnarray}
		for some constants $ {n_1},{n_2},n $ and $ s_0 $, with $ {{n_1}}^2 - {{n_2}}^2 + {n^2} = 1 $. 
		\item  $ < p,{N_\alpha}(s) > = 0. $
		\item  $ < p,{B_\alpha}(s) > = {\sigma} $ for some constant ${\sigma}$.
		\item $< p,{\alpha}(s) >  = {{m_1}}\sinh \left( {s + {s_0}} \right) + {{m_2}}\cosh \left( {s + {s_0}} \right)$ for some constants $ {m_1},{m_2} $ and $ s_0 $ such that \mbox{${m_2}^2 - {m_1}^2  \le 1 $}.
		\item Distance function in $ \mathbb{S}_1^3 $ between $ p $ and $ \alpha(s) $, $ \eta(s)=d(p,\alpha(s))$, satisfies
		\begin{eqnarray*}
			\cos(\eta(s))={{k_1}}\sinh \left( {s + {s_0}} \right) + {{k_2}}\cosh \left( {s + {s_0}} \right),
		\end{eqnarray*}
		for some constants $ {k_1},{k_2} $ and $ s_0 $ such that $ {{k_2}}^2 - {{k_1}}^2  \le 1 $.
	\end{enumerate}
\end{theorem}

\begin{proof}
	Firstly, let the statement (i) is valid. From Teorem \ref{teo3_01}, we say that $\alpha:I\to M\subset\mathbb{S}_{1}^3,\,\,\alpha(s) = \Phi(u(s),v(s)) $ be a geodesic in timelike conical surface $ M$ which is given by the parametrization (\ref{eq3_01}) such that some differentiable functions $ u=u(s) $ and $ v=v(s) $ which are satisfying (\ref{eq3_13})-(\ref{eq3_15}). Then, we obtain $ < p,{\alpha}(s) > =\cos(v(s)) $ and we get statement (v) by using (\ref{eq3_17}) and (\ref{eq3_21}).
	
	Now, let the statement (v) be valid. By using (\ref{eq2_06}) and take into consideration hypothesis, then we obtain 
	\begin{eqnarray*}
		{\kappa_g}(s)\left\langle {p,{N_\alpha}(s)} \right\rangle  + \left\langle {p,\alpha(s)} \right\rangle = \left\langle {p,\nabla _{{T_\alpha}(s)}^0{T_\alpha}(s)} \right\rangle = {{m_1}}\sinh (s + {s_0}) + {{m_2}}\cosh (s + {s_0}) = \left\langle {p,\alpha(s)} \right\rangle.
	\end{eqnarray*}
	From this equation, it must be $\left\langle{p,{N_\alpha}(s)} \right\rangle  = 0$, since $ \kappa_{g} \neq 0$. Thus, we obtain the statement (iii).
	
	Now, let the statement (iii) be valid. Let $ \beta_{s}(t) $ be the spacelike geodesic that passing through $ p=\beta_s(0) $ and $ \alpha(s)=\beta_{s}(v(s)) $ for some function $ v(s) $. From definition of spacelike geodesic in $ \mathbb{S}_{1}^3 $, we can write
	\begin{eqnarray*}
		\beta_{s}'(t)=f(t)p+g(t)\alpha(s),
	\end{eqnarray*}
	for some differentiable functions $ f(t) $ and $ g(t) $ which satisfy the condition $ f(t)^2+g(t)^2=1 $. Taking into account that hypothesis and $\left\langle{\alpha,{N_\alpha}(s)} \right\rangle  = 0$ by curve-hypersurface frame of $ \alpha $, then we obtain the rectifying condition
	\begin{eqnarray*}
		\left\langle {\beta_{s}'(v(s)),{N_\alpha}(s)} \right\rangle = 0.
	\end{eqnarray*}
	Namely, the statement (i) is obtained. Thus, we say that statements (i), (iii) and (v) are equivalent.
	
	Now, let us show that the statements (iii) and (iv) are equivalent. We suppose that the statement (iii) is valid. After using (\ref{eq2_06}) and hypothesis, we get 
	\begin{eqnarray}\label{eq3_23}
	\frac{d}{{ds}}\left\langle {p,{B_\alpha}(s)} \right\rangle {\rm{ }} = - {\tau_g}(s)\left\langle {p,{N_\alpha}(s)} \right\rangle  = 0.
	\end{eqnarray}
	then the statement (iv) is obtained. Conversely, let the statement (iv) is valid. if we take into consideration that hypothesis and $ \tau_g \neq0$, it easily seen that $ \left\langle{p,{N_\alpha}(s)} \right\rangle  = 0$ by using (\ref{eq3_23}). Hence, we see that $ ({\rm{iii}}) \Leftrightarrow ({\rm{iv}}) $.
	
	Now, let us show that the statements (v) and (vi) are equivalent. Since the position vector of $ \alpha(s) $ and the point $ p \in \mathbb{S}_1^3 $ are spacelike, in without lost of generality, we may write 
	\begin{eqnarray*}
		\left\langle {p,\alpha(s)} \right\rangle  = \cos (\eta(s)),
	\end{eqnarray*}
	for some function $ \eta(s) $. Thus, it is easily seen that $ ({\rm{v}}) \Leftrightarrow ({\rm{vi}}) $.  
	
	Finally, let us show that the statements (i) and (ii) are equivalent.
	Now, let the statement (i) be valid. In this case, we say that the statements (iii), (iv) and (v) are hold from the previous results. If we take derivative in (v), we get
	\begin{eqnarray*}
		\left\langle {p,{T_\alpha}(s)} \right\rangle  = {{m_1}}\cosh (s + {s_0}) + {{m_2}}\sinh (s + {s_0}),
	\end{eqnarray*}
	and so the equation (\ref{eq3_ek02}) is obtained. Moreover, we have
	\begin{eqnarray*}
		|p^{\bot}|^{2}&=&\left\langle p,N_\alpha \right\rangle^{2}	+\left\langle p,B_\alpha \right\rangle^{2}=\sigma^{2}	
	\end{eqnarray*}
	where $ {p^ \bot } \in Sp\left\{ {{N_\alpha},{B_\alpha}} \right\} $ by using the statements (iii) and (iv). Also, taking into account that $ T_\alpha $ is timelike, it is easily seen that
	\begin{eqnarray*}
		1 = \left\langle {p,p} \right\rangle {\rm{ }} = {\left\langle {p,\alpha} \right\rangle ^2} - {\left\langle {p,{T_\alpha}} \right\rangle ^2} + {\left\langle {p,{B_\alpha}} \right\rangle ^2} = {m_2}^2 - {m_1}^2 + {{\sigma}^2},
	\end{eqnarray*}
	for some constants $m_1=n_2,\,m_2=n_1$ and ${\sigma}=n$. Thus the equation (\ref{eq3_ek03}) is obtained. Namely, we see that $ ({\rm{i}}) \Rightarrow ({\rm{ii}}) $. Conversely, let the statement (ii) is valid. After integrating the equation (\ref{eq3_ek02}), we have
	\begin{eqnarray}\label{eq3_24}
	< p,{\alpha}(s) >  = {{n_1}}\cosh \left( {s + {s_0}} \right) + {{n_2}}\sinh \left( {s + {s_0}} \right) + c_0,
	\end{eqnarray}
	for some constant $ c_0 $. Moreover, we have
	\begin{eqnarray}\label{eq3_25}
	{{n_1}}^2 - {{n_2}}^2 = {\left\langle {p,\alpha} \right\rangle ^2} - {\left\langle {p,{T_\alpha}} \right\rangle ^2} = 1 - {n^2}.
	\end{eqnarray}
	by using the equation (\ref{eq3_ek03}). By considering together  the equations (\ref{eq3_ek02}), (\ref{eq3_24}) and (\ref{eq3_25}), then it is easily seen that $ c_0 =0$. Hence, the statement (v) is obtained and so $ ({\rm{v}}) \Rightarrow ({\rm{i}}) $. Consequently, we show that $ ({\rm{i}}) \Leftrightarrow ({\rm{ii}}) $.
\end{proof}

Now, we give the following theorem which characterizes all timelike rectifying curve in  $ \mathbb{S}_1^3 $.

\begin{theorem}\label{teo3_04}
	Let $ \alpha $ be a timelike non-planar curve in $ \mathbb{S}_1^3 $. Then, $ \alpha $ is a timelike rectifying curve
	if and only if, up to reparametrization, it is given by
	\begin{eqnarray}\label{eq3_26}
	\alpha(t) = \exp_p(\eta(t){\gamma(t)}) = \cos(\eta(t)) p+\sin(\eta(t)) \gamma(t),
	\end{eqnarray}
	where $ p $ is the fixed point in $ \mathbb{S}_1^3 $ such that $ p \notin {\rm Im}(\alpha) $, $ \gamma=\gamma(t) $ is a timelike unit speed curve in $ \mathbb{S}_1^2 \subset {T_p}\mathbb{S}_1^3 $, and $ \eta(t)=\arctan(a\,{\mathop{\rm sech}\nolimits} (t + t_0)) $ for some constants $ a\neq 0$ and $ t_0 $. 
\end{theorem}

\begin{proof}
	Let $ p \in \mathbb{S}_1^3 $ be a fixed point, $ \eta=\eta(t) $ be a positive function and $ \gamma=\gamma(t) $ be an unit speed timelike curve in $ \mathbb{S}_1^2 \subset {T_p}\mathbb{S}_1^3 $. If we take as $ \alpha(t) = \exp_p(\eta(t){\gamma(t)}) $ where $ p \notin {\rm Im}(\alpha) $, then the timelike unit tangent vector field ${T_\alpha}$ of $ \alpha $ is \begin{eqnarray}\label{eq3_28}
	{T_\alpha} = \frac{{\alpha'}}{\left\| {\alpha'} \right\|} = \frac{{ - \eta'\sin (\eta)}}{\left\| {\alpha'} \right\|}p + \frac{{\eta'\cos (\eta)}}{\left\| {\alpha'} \right\|}\gamma + \frac{{\sin (\eta)}}{\left\| {\alpha'} \right\|}\gamma',
	\end{eqnarray}
	where
	\begin{eqnarray*}
		\alpha'= - \eta'\sin (\eta)p + \eta'\cos (\eta)\gamma + \sin (\eta)\gamma',
	\end{eqnarray*}
	and
	\begin{eqnarray}\label{eq3_27}
	{\left\| {\alpha'} \right\|}^2={\sin ^2}(\eta) - {(\eta')^2} > 0.
	\end{eqnarray}
	Moreover, let $ s=s(t) $ be the arc length parameter of $ \alpha $ such that $ \left\| {\alpha'} \right\|(t)=s'(t) $. Then we have $ \left( {\frac{T_{\alpha}'}{{\left\| {\alpha'} \right\|}} - \alpha} \right)(t) = ({\kappa_g}{N_\alpha})(s)$ by using (\ref{eq2_06}). It means that $ N_\alpha $ is parallel to the spacelike  vector field $ \left( {\frac{T_{\alpha}'}{{\left\| {\alpha'} \right\|}} - \alpha} \right)
	$.	
	However, let $ \left\{ {\gamma,\gamma',{N_\gamma}} \right\} $ be Sabban frame (curve-surface frame) of the timelike unit speed curve $ \gamma $ in $\mathbb{S}_1^2 \subset {T_p}\mathbb{S}_1^3 $ where geodesic curvature of $ \gamma $ is defined by $ {\kappa_\gamma} = \det \left( {\gamma,\gamma',\gamma'',p} \right) $. Then, by using (\ref{eq3_05}) and Gauss formula, we have
	\begin{eqnarray}\label{eq3_29}
	\gamma'' = \gamma + {\kappa_\gamma}{N_\gamma}
	\end{eqnarray}
	where spacelike principal normal $ N_\gamma=\gamma \wedge \gamma' $ is tangent to $ \mathbb{S}_1^2 $, but normal to $ p $ and $ \gamma $. We get
	\begin{eqnarray*}
		< p,\frac{1}{\left\| {\alpha'} \right\|}{T_{\alpha}'} - \alpha >  = \frac{1}{\left\| {\alpha'} \right\|}{\left( {\frac{{\eta'}}{\left\| {\alpha'} \right\|}\sin (\eta)} \right)^\prime } + \cos (\eta).
	\end{eqnarray*}
	by using (\ref{eq3_28}) and (\ref{eq3_29}). According to the Theorem \ref{teo3_03}, $ \alpha $ is a timelike rectifying curve in $ \mathbb{S}_1^3 $ if and only if $  < p,{N_\alpha} >  = 0 $. Thus, it must be $ \frac{1}{\left\| {\alpha'} \right\|}{\left( {\frac{{\eta'}}{\left\| {\alpha'} \right\|}\sin (\eta)} \right)^\prime } + \cos (\eta)=0 $. After basic calculations, we reach to the differential equation
	\begin{eqnarray}\label{eq3_30}
	\sin (\eta)\eta'' - 2\cos (\eta){{(\eta')}^2} + \cos (\eta){{\sin }^2}(\eta)=0,
	\end{eqnarray}
	since $ N_\alpha $ is parallel to $  \left( {\frac{1}{\left\| {\alpha'} \right\|}{T_{\alpha}'} - \alpha } \right)$. Now, we consider a differentiable function $ h=h(t) $ such that $ \eta(t)=\arctan(h(t)) $ for solving the equation (\ref{eq3_30}). In that case, we reach to the equation
	\begin{eqnarray*}
		\frac{1}{{{{(1 + {h^2})}^{3/2}}}}(h\,h'' - 2{(h')^2} + {h^2}) = 0.
	\end{eqnarray*}
	The nontrivial solutions of this differential equation are given by the function $h(t)=a\sech (t+t_0)$ for some constants $ a \neq 0 $ and $ t_0 $. Thus, we have that $\alpha(t) = \exp_p(\eta(t){\gamma(t)})$ is a timelike rectifying curve in $ \mathbb{S}_1^3 $ iff $ \eta(t)=\tan^{-1}(a \sech(t+t_0)) $ for some constants $ a \neq 0 $ and $ t_0 $.
\end{proof}	

Chen and Dillen give some characterizations for rectifying curves with the viewpoint of extremal curves in Euclidean 3-space \cite{Chen2005}. Riemannian viewpoint of this idea is introduced by Lucas and Yag\"ues in Minkowski model of hyperbolic 3-space as Riemannian space form with negative constant curvature \cite{luc2016}.	

Now, we give some characterizations for Lorentzian version of timelike rectifying curves from the viewpoint of extremal curves in $ \mathbb{S}_1^3 $ which is Lorentzian space form with positive constant curvature 1. 

\begin{definition}\label{def3_03}
	Let $ \alpha $ be a timelike curve in $ \mathbb{S}_1^3 $, is given by $\alpha(t) = \exp_p(\eta(t){\gamma(t)})$ where $ p \in \mathbb{S}_1^3  $, $\eta(t)\neq0$ is an arbitrary function and $ \gamma(t) $ is a timelike curve lying in $\mathbb{S}_1^2 \subset {T_p}\mathbb{S}_1^3 $. Then $ \gamma $ is called the timelike pseudo-spherical projection of $ \alpha $.
\end{definition}

The following characterization means that a timelike rectifying curve in $ \mathbb{S}_1^3 $ is actually an extremal curve which assumes the the minimum value of the function $ \frac{{{\left\| {\alpha'} \right\|^4}\kappa_g^2}}{{{\sin}^2}(\eta)} $ at each point among the curves with the same timelike pseudo-spherical projection.

\begin{theorem}\label{teo3_05}
	Let $ p$ be a fixed point in $\mathbb{S}_1^3 $ and $ \gamma=\gamma(t) $ be a timelike unit speed curve with geodesic curvature $ \kappa_\gamma $ in $ \mathbb{S}_1^2 \subset {T_p}\mathbb{S}_1^3 $. Then, for any nonzero function  $  \eta(t) $, the geodesic curvature $ \kappa_g $ of a timelike regular curve $ \alpha $ in $\mathbb{S}_1^3$ which is given by $ \alpha(t) = \exp_p(\eta(t){\gamma(t)}) $, and $ \kappa_\gamma $ satisfy the inequality
	\begin{eqnarray} \label{eq3_30a}
	\kappa_\gamma^2 \le \frac{{{\left\| {\alpha'} \right\|^4}\kappa_g^2}}{{{\sin}^2}(\eta)},
	\end{eqnarray}
	with the equality sign holding identically if and only if $ \alpha $ is a timelike rectifying curve in $ \mathbb{S}_1^3 $.	
\end{theorem}

\begin{proof}
	Let $ \eta $ be a nonzero function and $ \alpha $ be a timelike regular curve in $\mathbb{S}_1^3$ which is given by $ \alpha(t) = \exp_p(\eta(t){\gamma(t)}) $, where $ \gamma $ is a timelike unit speed curve in $ \mathbb{S}_1^2 \subset {T_p}\mathbb{S}_1^3 $. If we consider together the equation (\ref{eq3_28}) with definition the map $\exp_p$, and taking account that $N_\gamma$ is a spacelike vector orthogonal to timelike subspace  $S_p\left\lbrace p,\gamma,\gamma' \right\rbrace $, and so we find that  $N_\gamma$ is orthogonal to both $ \alpha $ and $ T_\alpha $. Then, we have
	\begin{eqnarray} \label{eq3_31}
	N_\gamma=\gamma\wedge \gamma'=\cos(\theta)N_\alpha+\sin(\theta)B_\alpha
	\end{eqnarray}
	such that $\theta=\theta(t)$ is a arbitrary function. 
	By differentiating of (\ref{eq3_31}) with respect to $ t $, in addition to applying (\ref{eq2_06}) and (\ref{eq3_29}), we obtain
	\begin{eqnarray}  \label{eq3_33}
	\kappa_{\gamma} \gamma'=(\left\| {\alpha'} \right\|\kappa_g \cos(\theta))T_\alpha+(\theta'+\left\| {\alpha'} \right\|\tau_g)(-\sin(\theta)N_\alpha + \cos(\theta)B_\alpha)
	\end{eqnarray}
	where  $\left\| {\alpha'} \right\|$ satisfies the eq. (\ref{eq3_27}). Since $\left\langle \gamma', \gamma'\right\rangle =-1$ and $\left\langle T_\alpha, T_\alpha\right\rangle=-1$, we have
	\begin{eqnarray} \label{eq3_34}
	\kappa_{\gamma}^2=(\left\| {\alpha'} \right\| \kappa_g \cos(\theta))^{2}-(\theta'+\left\| {\alpha'} \right\|\tau_g)^{2}.
	\end{eqnarray} 
	
	Now, we will give the point $ p $ with respect to the curve-hypersurface frame $\left\lbrace\alpha,T_\alpha,N_\alpha,B_\alpha \right\rbrace $ of $ \alpha $. By using (\ref{eq3_26}), we get
	\begin{eqnarray} \label{eq3_35}
	\left\langle p, \alpha\right\rangle=\cos(\eta),
	\end{eqnarray} 
	and after differentiating of (\ref{eq3_35}),
	\begin{eqnarray} \label{eq3_36}
	\left\langle p,T_\alpha \right\rangle =-\frac{\eta'\sin(\eta)}{\left\| {\alpha'} \right\|}
	\end{eqnarray}
	by using (\ref{eq3_28}). Now, suppose that
	\begin{eqnarray*}
		\gamma'={\sigma_1}T_\alpha+{\sigma_2}(-\sin(\theta)N_\alpha+\cos(\theta)B_\alpha)	
	\end{eqnarray*}
	such that
	\begin{eqnarray*}
		{\sigma_1}=\frac{{\left\| {\alpha'} \right\|{\kappa_g}\cos (\theta)}}{{{\kappa_\gamma}}},~ {\sigma_2}=\frac{{\theta' + \left\| {\alpha'} \right\|{\tau_g}}}{{{\kappa_\gamma}}}
	\end{eqnarray*}
	by using (\ref{eq3_33}). Since  $\left\langle p,\gamma'\right\rangle =0 $ and $\left\langle p,N_\gamma\right\rangle=0$, we obtain linear equation system depending on $ \left\langle p,N_\alpha\right\rangle $ and $ \left\langle p,B_\alpha\right\rangle $ by using (\ref{eq3_31}) and (\ref{eq3_36}) and so its solution is given by 
	\begin{eqnarray} \label{eq3_37}
	\left\langle {p,{N_\alpha}} \right\rangle  = \frac{{{\sigma_1}}}{{{\sigma_2}}}\frac{{\eta'}}{{\left\| {\alpha'} \right\|}}\sin (\theta)\sin (\eta),\\ \label{eq3_38}
	\left\langle {p,{B_\alpha}} \right\rangle  =  - \frac{{{\sigma_1}}}{{{\sigma_2}}}\frac{{\eta'}}{{\left\| {\alpha'} \right\|}}\cos (\theta)\sin (\eta).
	\end{eqnarray}
	From the equations (\ref{eq3_35})-(\ref{eq3_38}), we get
	\begin{eqnarray*}
		p=\cos(\eta)\alpha-\frac{\eta'\sin(\eta)}{\left\| {\alpha'} \right\|}T_\alpha+\frac{{\sigma_1}\eta'\sin(\theta)\sin(\eta)}{{\sigma_2}\left\| {\alpha'} \right\|}N_\alpha-\frac{{\sigma_1}\eta'\cos(\theta)\sin(\eta)}{{\sigma_2}\left\| {\alpha'} \right\|}B_\alpha.
	\end{eqnarray*}
	Then, 
	\begin{eqnarray*}
		1=\left\langle p,p \right\rangle= \cos^{2}(\theta)+\sin^{2}(\theta)\left( \frac{\eta'}{\left\| {\alpha'} \right\|}\right) ^{2}\left( \left(\frac{{\sigma_1}}{\sigma_2} \right)^{2}-1\right) 
	\end{eqnarray*} 
	and so it must be 
	\begin{eqnarray*}
		\left( \frac{\eta'}{\left\| {\alpha'} \right\|}\right)^{2} \left( \left( {\frac{{{\sigma_1}}}{{{\sigma_2}}}} \right)^{2}-1\right)=1.
	\end{eqnarray*}
	Also, if we put $ {\sigma_1} $ and $ {\sigma_2} $ in this equaiton, we reach to the equation
	\begin{eqnarray} \label{eq3_39}
	(\theta'+\left\| {\alpha'} \right\|\tau_g)^{2}=\frac{(\eta')^2}{\left\| {\alpha'} \right\|^{2}+(\eta')^2}(\left\| {\alpha'} \right\|\kappa_g \cos(\theta))^{2}.
	\end{eqnarray}
	
	After putting (\ref{eq3_39}) in (\ref{eq3_34}), and also considering the equation (\ref{eq3_27}), we get
	\begin{eqnarray*} \label{eq3_40}
		\kappa_{\gamma}^{2}=\frac{\left\| {\alpha'} \right\|^{4}{\kappa_g}^2 \cos^2(\theta)}{\sin^2(\eta)}
	\end{eqnarray*}
	which implies inequality (\ref{eq3_30a}). It is clear that, equality situation of (\ref{eq3_30a}), it must be $\sin(\theta)=0$. Also, take into consideration the equation (\ref{eq3_31}), we get  $N_\gamma=\mp N_\alpha$. It means that $N_\gamma//N_\alpha$. Then, by Theorem \ref{teo3_01}, $\alpha$ is a geodesic in timelike conical surface $ M$ which is given by the parametrization (\ref{eq3_01}) and so it is a timelike rectifying curve. Therefore, $ \alpha $ is  a timelike rectifying curve in $ \mathbb{S}_1^3$ iff equality situation of (\ref{eq3_30a}) is satisfied. 
\end{proof}

Now, we show that a timelike curve in  De Sitter 3-space, which has non-zero constant geodesic curvature and linear geodesic torsion, congruent to a timelike rectifying curve, which is generated by a spiral type unit speed timelike curve with certain geodesic curvature in 2-dimensional pseudo-sphere, and vice versa. Namely, the following corollary is a construction method for timelike spiral type rectifying curves in $ \mathbb{S}_1^3 $.

\begin{corollary}\label{cor3_01}
	A timelike regular curve $ \alpha $ is given by $ \alpha(s)=\exp_p(\eta(s)\gamma(s)) $ in $ \mathbb{S}_1^3 $ with nonzero constant geodesic curvature $ \kappa_0 $ and linear geodesic torsion
	\begin{eqnarray*}
		{\tau_g}(s) = {c_1}\sinh \left( {s + {s_0}} \right) + {c_2}\cosh \left( {s + {s_0}} \right)
	\end{eqnarray*}
	such that $ {{c_2}}^2 - {{c_1}}^2 - {\kappa_0}^2 < 0 $ for some constants $ {c_1} $, $ {c_2} $ and $ s_0 $ iff $ \alpha $ is congruent to a timelike rectifying curve which is generated by a unit speed timelike spiral type curve $ \gamma(t) $ in  $ \mathbb{S}_1^2 \subset {T_p}\mathbb{S}_1^3 $ with geodesic curvature
	\begin{eqnarray*}
		{\kappa_\gamma}(t) = {b} {({\cosh ^2}(t + {t_0}) + {a^2})^{-3/2}}
	\end{eqnarray*}
	for some constants  $ a\neq0,\,{b}\neq0 $ and $ t_0 $.
\end{corollary}

\begin{proof}
	Let  $\alpha = \exp_p(\eta\,{\gamma})$ be a timelike regular curve in $\mathbb{S}_{1}^3$ with nonzero constant geodesic curvature $ \kappa_0 $ and geodesic torsion $\tau_g(s)={c_1}\sinh(s+s_0)+{c_2}\cosh(s+s_0)$ in arc length parameter $ s  $ where  ${c_2}^{2}-{c_1}^{2}-\kappa_0^{2}<0$ for some constants ${c_1}$ , ${c_2}$ and $s_0$. Then, $\alpha$ is a timelike rectifying curve by Theorem \ref{teo3_02}. Hence, by using Theorem \ref{teo3_04}, we take $\eta(t)=\arctan(a \sech (t+t_0))$ for some constants $a\neq0$ and $t_0$. Also, taking account of Theorem \ref{teo3_05}, we obtain that $\kappa_{\gamma}(t)={b}(\cosh^2(t+t_0)+a^2)^{-3/2}$ for nonzero constant $b=a(1+a^2)\kappa_0$.
	
	On the other hand, let  $\alpha=\exp_p(\eta\,\gamma)$ be a timelike rectifying curve in $\mathbb{S}_{1}^3$ which is generated by a timelike unit speed curve $\gamma=\gamma(t)$ which is lying in $\mathbb{S}_{1}^2 \subset T_p{\mathbb{S}_{1}^3}$ with geodesic curvature $\kappa_{\gamma}(t)=b(\cosh^2(t+t_0)+a^2)^{-3/2}$ such that $ b \neq 0 $. Then, by using Theorem \ref{teo3_04}, we get function $\eta(t)=\arctan(a \sech (t+t_0))$ for some constants  $a\neq0$ and $t_0$, and so, we have
	\begin{eqnarray*}
		\frac{\left\| {\alpha'} \right\|^4}{\sin^2(\eta)}=\frac{a(1+a^2)^2}{(a^2+ \cosh^2(t+t_0))^3}.
	\end{eqnarray*}
	Since  $\alpha$ is a timelike rectifying curve, in accordance with Theorem \ref{teo3_05}, 
	\begin{eqnarray*}
		\kappa_{\gamma}^{2}=	\frac{\left\| {\alpha'} \right\|^4\kappa_g^{2}}{\sin^2(\eta )}.
	\end{eqnarray*}
	Thus, we obtain that the nonzero constant
	\begin{eqnarray*}
		\kappa_g^2=\frac{b^2}{a^2(1+a^2)^2}.
	\end{eqnarray*} 
	Finally, if we take into account Theorem \ref{teo3_02}, the proof is complete.
\end{proof}

\section{Some Examples of Non-Degenerate Rectifying Curves} \label{Sec:4}

Now, we give some examples for timelike or spacelike rectifying curve in $ {\mathbb{S}_{1}^3} $.

\begin{figure}[!ht]
	\centerline{\includegraphics[width=0.6\textwidth,keepaspectratio]{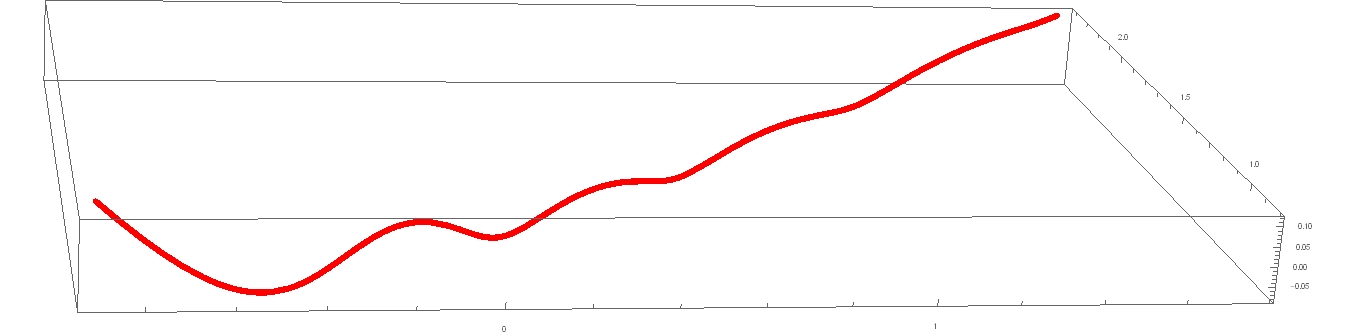}}
	\caption{Stereographic projection in Minkowski 3-space of timelike rectifying curve $ \alpha $}
	\label{fig5_05}
\end{figure}

\begin{figure}[!ht]
	\centering
	\subfloat[]{%
		\includegraphics[width=0.40\textwidth,keepaspectratio]{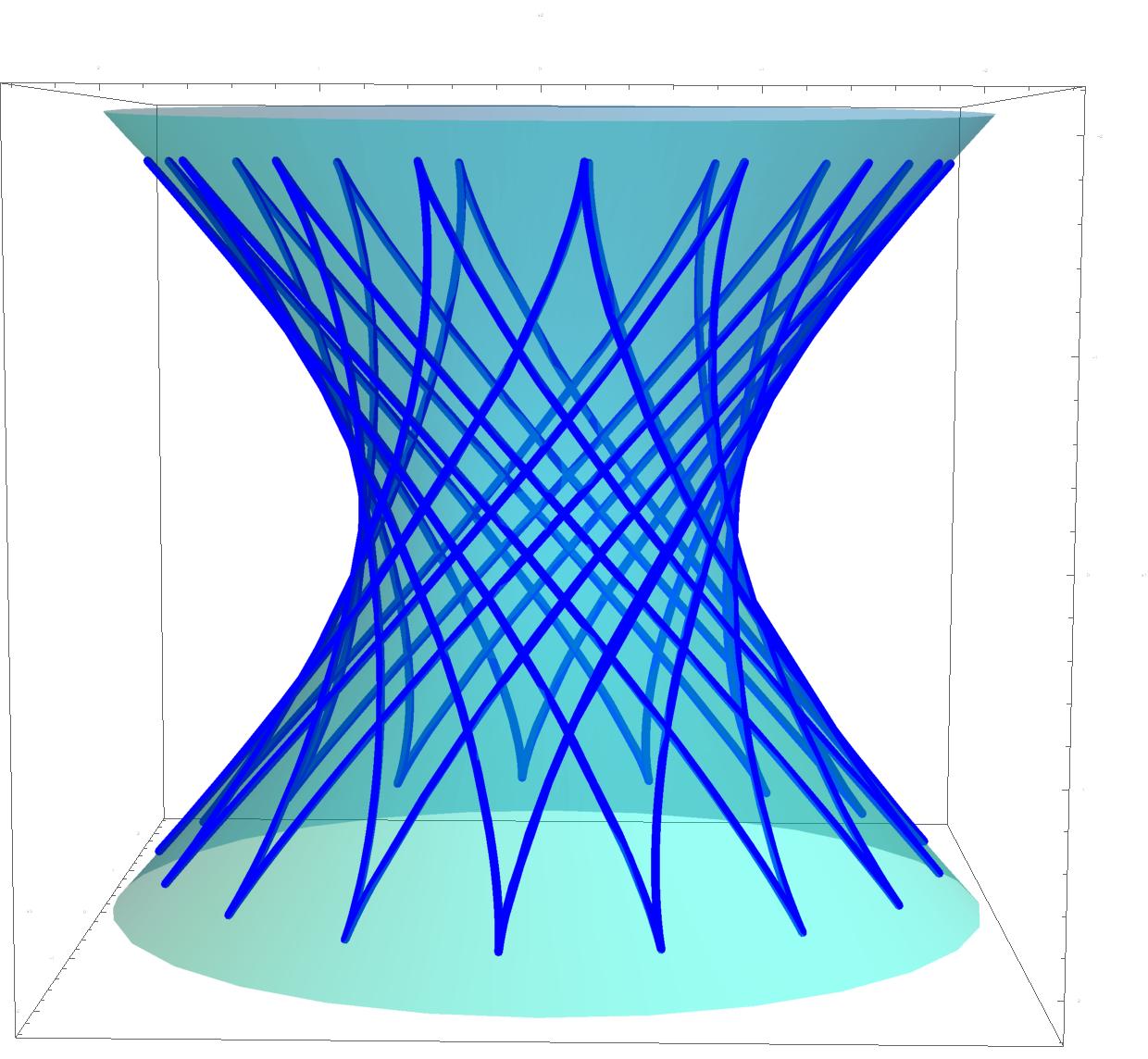}
		\label{fig5_01}}
	\quad
	\subfloat[]{%
		\includegraphics[width=0.35\textwidth,keepaspectratio]{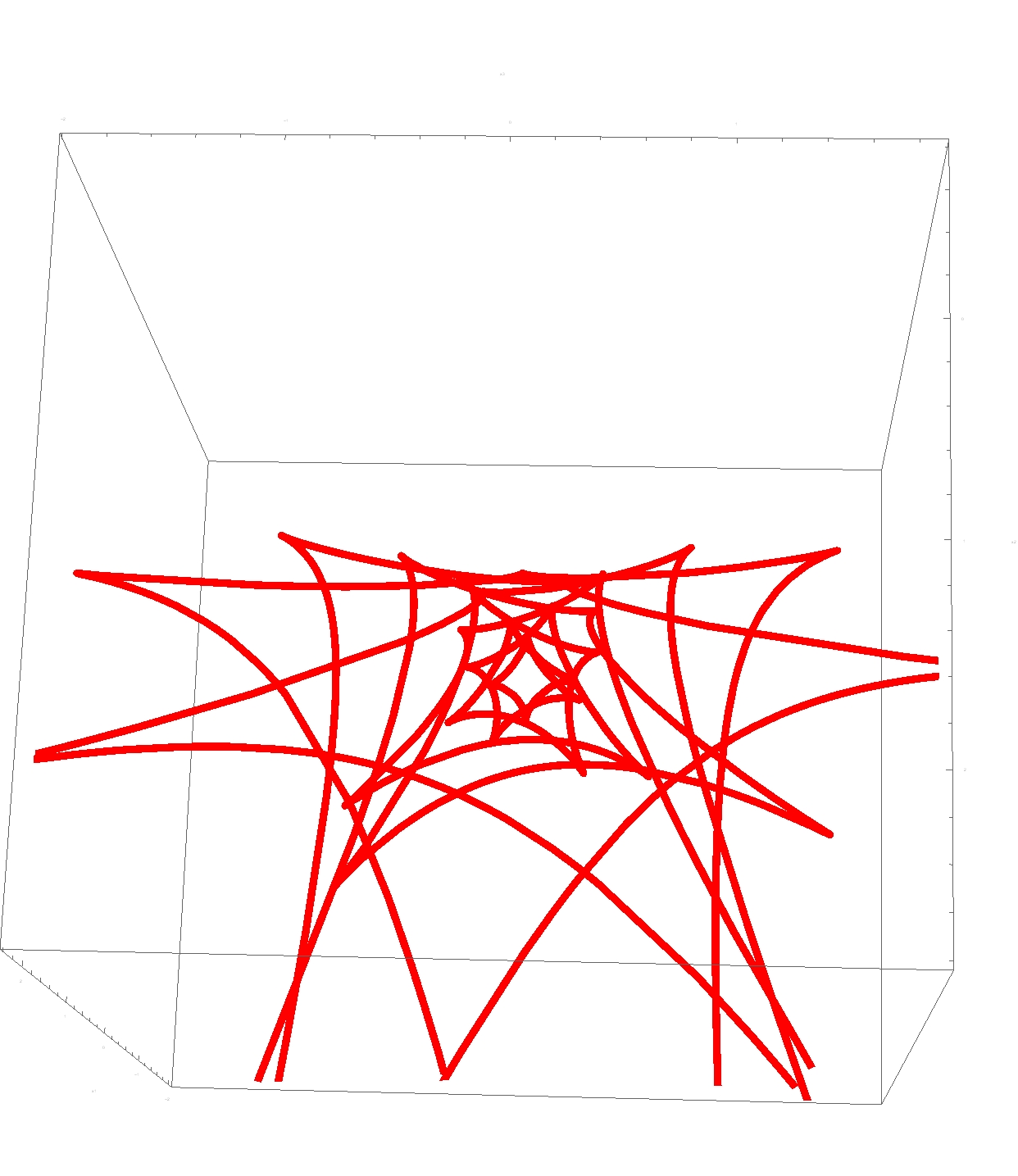}
		\label{fig5_02}}
	
	\caption{(a) Timelike pseudo-spherical projection curve $ \gamma(t) $ in $ {\mathbb{S}_{1}^2} $ of timelike rectifying curve $ \alpha $ \\ (b) Stereographic projection in Minkowski 3-space of timelike rectifying curve $ \alpha $ }
\end{figure}

\begin{figure}[!ht]
	\centering
	\subfloat[]{%
		\includegraphics[width=0.4\textwidth,keepaspectratio]{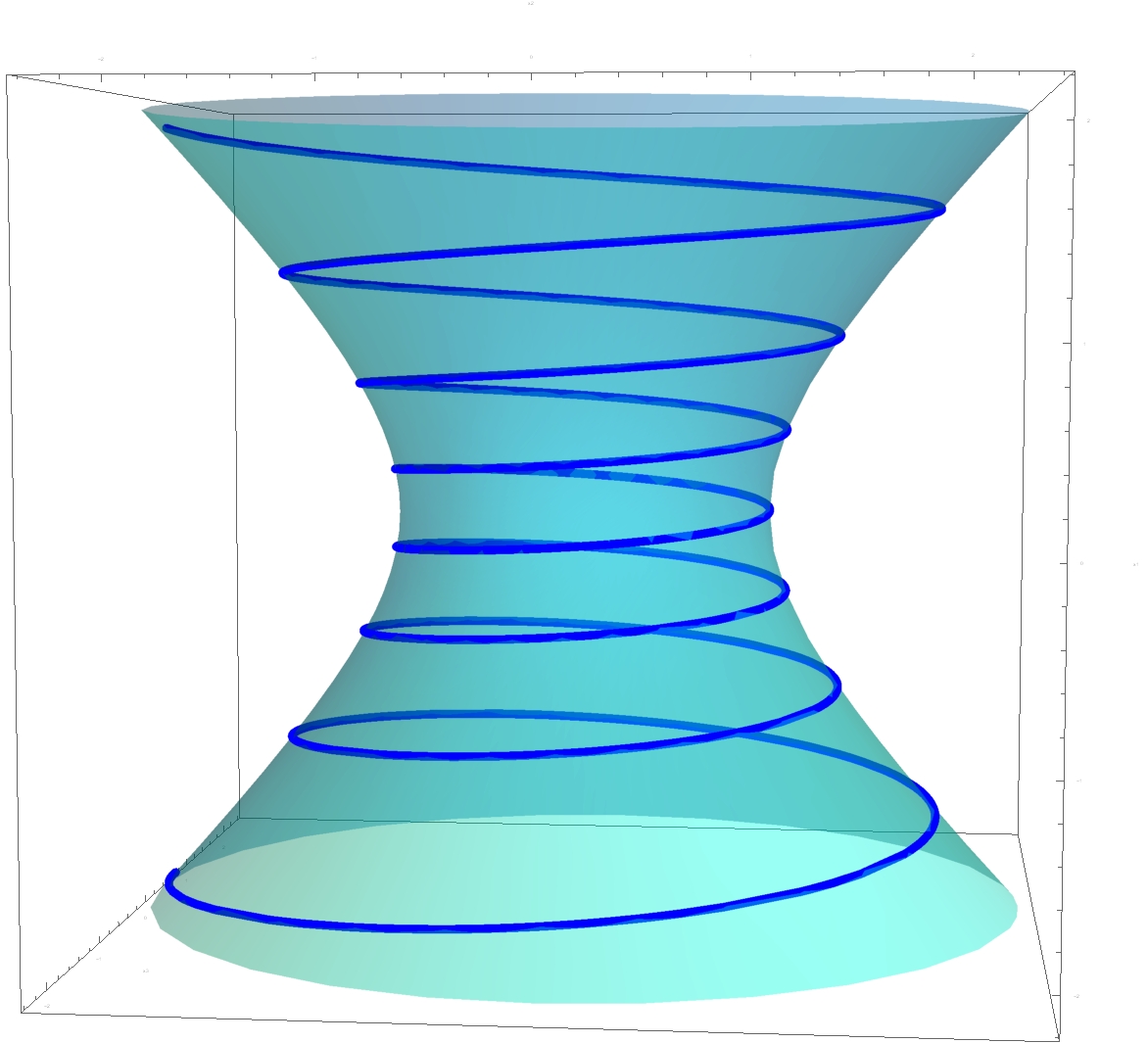}
		\label{fig5_03}}
	\quad
	\subfloat[]{%
		\includegraphics[width=0.4\textwidth,keepaspectratio]{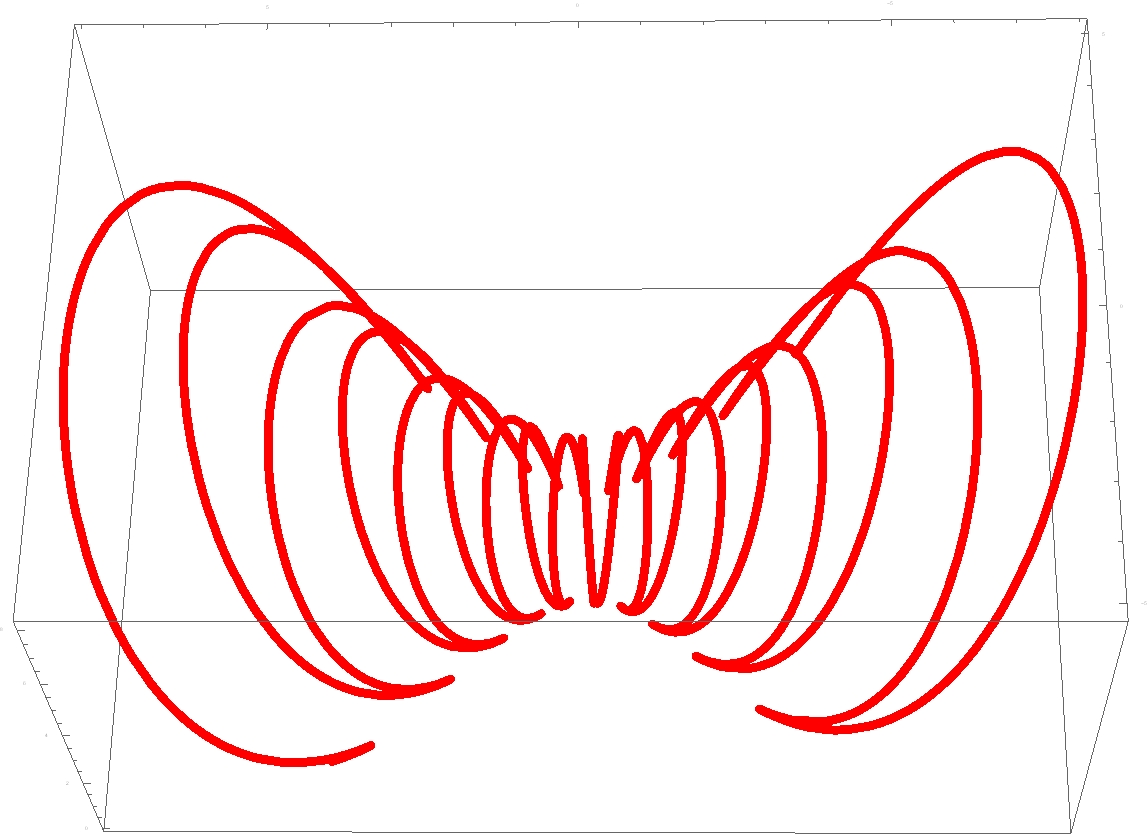}
		\label{fig5_04}}
	
	\caption{(a) Spacelike pseudo-spherical projection curve $ \gamma(t) $ in $ {\mathbb{S}_{1}^2} $ of spacelike rectifying curve $ \alpha $ \\ (b) Stereographic projection in Minkowski 3-space of spacelike rectifying curve $ \alpha $}
\end{figure}

\begin{example}
	Let $ \alpha $ be a  timelike rectifying curve in $ {\mathbb{S}_{1}^3} $ with geodesic curvature $ \kappa_{g}(s)= 10$ and geodesic torsion $ \tau_g(s)=2 \sinh(s)+2\cosh(s) $. Then, we obtain stereographic projection in Minkowski 3-space which is congruent to $ \alpha $ by using numeric methods in Mathematica (see Figure \ref{fig5_05})	
\end{example}

\begin{example}
	Let a timelike pseudo-spherical projection curve be given by
	\begin{eqnarray*}
		\gamma(t) = \left( {\frac{{15}}{8}\cos \left( {17t} \right),0,\frac{{25}}{{16}}\cos \left( {9t} \right) + \frac{9}{{16}}\cos \left( {25t} \right),\frac{{25}}{{16}}\sin \left( {9t} \right) - \frac{9}{{16}}\sin \left( {25t} \right)} \right),	
	\end{eqnarray*}
	in $\mathbb{S}_{1}^2 \subset T_p{\mathbb{S}_{1}^3}$ (see Figure \ref{fig5_01}). Then, the parametrization of timelike rectifying curve $ \alpha $ is
	\begin{eqnarray*}
		\alpha(t) = \frac{{{\mathop{\rm sech}\nolimits} (t)}}{{16\sqrt {1 + {{{\mathop{\rm sech}\nolimits} }^2}(t)} }}\left( {30\cos (17t)\,{\mathop{\rm sech}\nolimits} (t),\frac{{16}}{{{\mathop{\rm sech}\nolimits} (t)}},25\cos (9t) + 9\cos (25t),25\sin (9t) - 9\sin (25t)} \right)	
	\end{eqnarray*}
	where the point $p=(0,1,0,0)\in\mathbb{S}_{1}^3$ and the function $ \eta(t)=\arctan(\sech(t))$
	(see Figure \ref{fig5_02}).
\end{example}

\begin{example}
	Let a spacelike pseudo-spherical projection curve be given by
	\begin{eqnarray*}
		\gamma(t) = \left( {\sinh (\frac{t}{{15}}),\,\,\cosh (\frac{t}{{15}})\cos (t),\,\,\cosh (\frac{t}{{15}})\sin (t),0} \right),	
	\end{eqnarray*}	
	in $\mathbb{S}_{1}^2 \subset T_p{\mathbb{S}_{1}^3}$ (see Figure \ref{fig5_03}). Then, the parametrization of timelike rectifying curve $ \alpha $ is
	\begin{eqnarray*}
		\alpha(t) = \frac{1}{{\sqrt {1 + {{\sec }^2}(t)} }}\left( {\sec(t)\sinh (\frac{t}{{15}}),\cosh (\frac{t}{{15}}),\cosh (\frac{t}{{15}})\tan (t),1} \right)	
	\end{eqnarray*}
	where the point $p=(0,0,0,1)\in\mathbb{S}_{1}^3$ and the function $ \eta(t)=\arctan(\sec(t))$ (see Figure \ref{fig5_04}).
\end{example}

\end{document}